\documentclass[11pt,reqno]{amsart}
\usepackage{mathrsfs}
\usepackage{}

\def\subjclass#1{{\renewcommand{\thefootnote}{}%
\footnote{\emph{Mathematics Subject Classification (2010):} #1}}}

\usepackage{amsfonts}
\usepackage{amssymb}
\DeclareMathOperator{\curl}{curl} \DeclareMathOperator{\dv}{div}

\date{\today}

\hoffset=- 2cm \voffset=- 0cm 
 \theoremstyle{plain}
\newtheorem{Thm}{Theorem}

\newtheorem{Lem}[Thm]{Lemma}
\newtheorem{Cor}[Thm]{Corollary}

\newtheorem{Def}[Thm]{Definition}

\usepackage{amssymb}
\usepackage{color}

\newcommand {\p}{\partial}
\newcommand{\q}{\quad}
\newcommand{\qq}{\qquad}

\def\lam{\lambda}

\def\O{\Omega}

\def\F{\mathbf F}
\def\A{\mathbf A}

\def\h{\mathbf h}
\def\u{\mathbf u}
\def\w{\mathbf w}

\def\v{\vskip}

\errorcontextlines=0 \numberwithin{equation}{section}
\numberwithin{Thm}{section}

\begin{document}
\large

\title[Lorentz space estimates for vector fields]
{Lorentz space estimates for vector fields with divergence and curl in Hardy spaces}

\author[]{Yoshikazu Giga}
\author[]{Xingfei Xiang}

\address{Yoshikazu Giga:
Graduate School of Mathematical Sciences,
University of Tokyo, Tokyo 153-8914, Japan; }
\email{labgiga@ms.u-tokyo.ac.jp}

\address{Xingfei Xiang: Department of Mathematics,
Tongji University, Shanghai 200092, P.R. China; }
\email{xiangxingfei@126.com}

\thanks{ }



\begin{abstract}
In this note, we establish the estimate
on the Lorentz space $L(3/2,1)$ for vector
fields in  bounded domains under the
assumption  that the normal or the tangential
component of the vector fields on the
boundary vanishing. We prove that the $L(3/2,1)$
norm of the vector field can be controlled
by the norms of its divergence and curl in
the atomic Hardy spaces and the $L^1$ norm
of the vector field itself.

\end{abstract}
\maketitle

\section{Introduction}\label{section1}

In this note, we consider the estimate on the Lorentz space $L(3/2,1)$
for vector fields in a bounded domain  $\O$ in $\mathbb R^3$ by assuming
that the divergence and the curl in
 atomic Hardy spaces.
 This work originates from  the problem
  raised by Bourgain and Brezis in \cite[open problem 1]{BB},
 where assume the divergence-free and the curl in $L^1$.
Our result in this note shows that the $L(3/2,1)$ estimate controlled by the
 divergence and curl in
 the atomic Hardy spaces  holds
if the normal or the tangential component of the vector fields
 on the boundary vanishing.
  While for the case where
the divergence-free and the curl in $L^1$ we still don't know how to treat.

For the case where the vector $\u\in W_{0}^{1,p}(\O)$ with $1<p<\infty,$
the estimate on the Lorentz space with the divergence and the curl in
the Hardy spaces is easy to obtain.
Indeed, applying the representation
 from the fundamental theorem of vector analysis
 \begin{equation}\label{1.3}
 \u=-\frac{1}{4\pi}\mathrm{grad}\int_{\O}\frac{1}{|x-y|}\dv\u(y)
dy+ \frac{1}{4\pi}\curl\int_{\O}\frac{1}{|x-y|}\curl\u(y)
dy,
 \end{equation}
we can obtain the estimate on $\nabla\u$ by the estimate on
the singular integrals in Hardy spaces
 (see \cite[Theorem 3.3]{ST})
 $$\|\nabla\u\|_{L^{1}(\O)}\leq C(\|\dv\u\|_{\mathcal {H}(\O)}+
 \|\curl\u\|_{\mathcal {H}(\O)}),
 $$
 where the norm $\|\cdot\|_{\mathcal {H}(\O)}$ denotes
 $$\|f\|_{\mathcal {H}(\O)}=\|\tilde{f}\|_{\mathcal {H}(\mathbb R^3)},
 $$
  $\tilde{f}$ is the zero extension of the function $f$ outside
  of $\O$ and $\mathcal {H}(\mathbb R^3)$ is
  the usual Hardy space $H^p(\mathbb R^3)$
  with $p=1$ (see \cite[Chapter III]{ST});
  here and hereafter $C$ denotes a positive constant independent
  of vector fields or functions and its numerical value may be
  different in each occasion.
  Then using the $L^1$ estimate for the Newtonian potential,
  and noting that $\mathcal {H}(\O)$ is continuously imbedded into
  the space $L^{1}(\O)$,
  we have
 $$
 \aligned
 \|\u\|_{L^{1}(\O)}&\leq C(\|\dv\u\|_{L^1(\O)}+
 \|\curl\u\|_{L^1(\O)})\\
 &\leq C(\|\dv\u\|_{\mathcal {H}(\O)}+
 \|\curl\u\|_{\mathcal {H}(\O)}).
 \endaligned
 $$
Noting the fact that $W^{1,1}(\O)$ is continuously imbedded into
the Lorentz space $L(3/2,1),$ we can thus obtain that
 $$\|\u\|_{L(3/2,1)(\O)}\leq C(\|\dv\u\|_{\mathcal {H}(\O)}+
 \|\curl\u\|_{\mathcal {H}(\O)}).
 $$
But for the  vector $\u$ not
 vanishing on the boundary, two terms involving boundary integrals
 will be added to the representation \eqref{1.3}
 (see \cite{W1}). Both of the terms are not easy to deal
 with on the Lorentz space.

 This note studies the vector fields with the normal or
 the tangential components  on the boundary
 vanishing but not the zero boundary condition,
  in contrast to the representation \eqref{1.3}, the
Helmholtz-Weyl decomposition on the Lorentz spaces in
our proof will be employed.

Let $\nu(x)$ be the unit outer normal vector at
$x\in \p\O.$ Our main result now reads:

\begin{Thm}\label{Thm}
Assume $\O$ is a bounded domain in $\Bbb R^3$ with $C^2$ boundary.
Let $\u\in C^{1,\alpha}(\bar{\O})$
with $\dv\u\in \mathcal {H}_{\mu_0}(\O)$ and
$\curl\u\in \mathcal {H}_{\mu_0}(\O),$
where the atomic Hardy space $\mathcal {H}_{\mu_0}(\O)$ is
defined in Definition \ref{def2.4}.
Then if
either $\nu\cdot\u=0$ or $\nu\times\u=0$ on $\p\O,$  we have
\begin{equation}\label{estimate}
\| \u\|_{L(3/2,1)(\O)}
\leq C\left(\|\dv\u\|_{\mathcal {H}_{\mu_0}(\O)}
+\|\curl\u\|_{\mathcal {H}_{\mu_0}(\O)}+\|\u\|_{L^1(\O)}\right),
\end{equation}
where the constant $C$ depending only on $\mu_0$ and the domain $\O$,
but not on the vector $\u$.
\end{Thm}

Our proof for proving Theorem \ref{Thm} is based on the
Helmholtz-Weyl decomposition on the Lorentz spaces
(For the decomposition on $L^p$ spaces we refer to \cite[Theorem 2.1]{KY}):
\begin{equation}\label{1.5}
\u=\nabla p_{\u}+\curl\mathbf w_{\u}+\mathscr{H}_{\u},
\end{equation}
where $\mathscr{H}_{\u}$ is the harmonic part,
the function $p_{\u}$ satisfies the Laplace equation and the vector
$\mathbf w_{\u}$ satisfies the elliptic system involving  $\curl\u$.
The advantage of using \eqref{1.5} is that
 we need not handle the terms involving boundary integral.
Then  our strategy is to establish the estimate on the Lorentz norm of
$\nabla p_{\u}$ by
the norm of $\dv\u$ in the Hardy space, and of $\curl\mathbf w_{\u}$ by
the norm of $\curl\u$. To obtain these estimates,
 the duality between several spaces will be introduced.
The harmonic part $\mathscr{H}_{\u}$ because of its
regularity can be controlled by the
$L^1$ norm of the vector $\u$ itself.

We would like to mention that starting with the pioneering work in \cite{BB1}
by Bourgain and Brezis, related interesting $L^1$ estimate for
vector fields have been well studied by
several authors, see [2-5, 14-16, 20-23, 25] and the references therein.
In particular, Bourgain and Brezis in \cite{BB2, BB}
showed  the $L^{3/2}(\mathbb{T}^3)$ norm of the
divergence-free vector $\u$ can be controlled by the
$L^1$ norm of $\curl \u;$   Lanzani and  Stein in \cite{LS}
obtained the estimate of the smooth $q$-forms on the $L^{3/2}$ spaces
 by the $L^1$ norms of their exterior
derivative and co-exterior derivative;
I. Mitrea and M. Mitrea in \cite{MM} considered these estimates
in  homogeneous Besov spaces;
Van Schaftingen in \cite{VS4} established the estimates
in Besov, Triebel-Lizorkin and Lorentz spaces
of differential forms on $\mathbb R^n$ in terms of their $L^1$ norm.

The organization of this paper is as follows. In Section 2, some
well known spaces are introduced.
The proof of Theorem \ref{Thm} will be given in section 3.
Throughout the paper, the bold typeface is used to indicate vector
quantities; normal typeface will be used for
vector components and for scalars.

\section{Lorentz spaces, Hardy spaces and BMO spaces}\label{section2}

In this section we will introduce several well-known spaces
and show some properties of these spaces,
 which will be used in the proof of our theorem.
These spaces can be found in many literatures and papers.

\subsection{Lorentz spaces}
Let $(X,S,\mu)$ be a $\sigma-$finite measure space and $f:X\to\mathbb{R}$ be
a measurable function. We define the distribution function of $f$ as
$$f_{*}(s)=\mu(\{|f|>s\}),\q s>0,
$$
and the nonincreasing rearrangement of $f$ as
$$f^{*}(t)=\inf\{s>0, f_{*}(s)\leq t\},\q t>0.
$$
The Lorentz space is defined by
$$L(m,q)=\left\{f: X\to\mathbb{R} \text{ measurable},
\|f\|_{L^{m,q}}<\infty\right\} \q \text{with } 1\leq m<\infty
$$
equipped with the quasi-norm
$$\|f\|_{L(m,q)}=\Big(\int_0^{\infty}\left(t^{1/m}f^{*}(t)\right)^q\frac{dt}{t}\Big)^{1/q},
\q 1\leq q<\infty
$$
and
$$\|f\|_{L(m,\infty)}=\sup_{t>0} t^{1/m} f^{*}(t),\q q=\infty.
$$
From the definition of the Lorentz space, we can obtain the following properties.
\begin{Lem}\cite[Embedding theorem]{AD}\label{Lem2.1}
Let $\O$ be a bounded domain in $\Bbb R^3.$  We have the following conclusions.
\begin{itemize}
\item[(i)] If $p>q$ and $1<m<\infty,$ then
$\|f\|_{L(m,p)(\O)}\leq C \|f\|_{L(m,q)(\O)}.$
\item[(ii)] If $m<n$ and $1\leq p, q\leq\infty,$ then
$\|f\|_{L(n,p)(\O)}\leq C(\O) \|f\|_{L(m,q)(\O)}.$
    \end{itemize}
\end{Lem}
\begin{Lem}\cite[Duality of Lorentz spaces]{CW,HU}\label{Lem2.2}
Assume that  $\O$ is a bounded domain in $\Bbb R^3,$ and let $f$ be a bounded
function belonging to $L(m,p)$ with $1<m<\infty$ and $1\leq p\leq \infty.$
Then there exists a constant $C$ independent of $f$ satisfying
$$\|f\|_{L(m,p)}\leq
 C \sup_{g\in{L^{m',p'}}}\frac{\int_{\O} f g dx}{\|g\|_{L(m',p')}}
$$
with $m'=m/(m-1),$ $p'=p/(p-1).$
\end{Lem}
\begin{proof}
For the case $1\leq p<\infty,$
the dual space of $L(m,p)$ is $L(m',p')$ (see \cite{HU}),
the conclusion follows immediately.
The dual space of  $L(m,\infty)$ is $L(m',1)\oplus S_0\oplus S_{\infty},$
where the spaces $S_0$ and $S_{\infty}$ defined in \cite{CW} annihilate
all functions which are bounded and supported
on any set of finite measure (see \cite{CW}). Thus,
the desired estimate follows from the assumptions when $p = \infty$.
\end{proof}

\subsection{Hardy spaces}
There are several equivalent definitions for Hardy spaces
in $\mathbb R^n$ and in bounded domains.
In this paper we define the Hardy space in the bounded
domains by the atomic decomposition.

 \begin{Def}\label{def2.3}
 An $\mathcal H_{\mu_0} (\mu_0>0)$ atom with respect to the cube $Q$
  is a function $a(x)$ satisfying the following three conditions:
 \begin{itemize}
\item[(i)]  the function $a(x)$ is supported in a cube $Q$;
\item[(ii)] the inequality $|a|\leq |Q|^{-1}$ holds almost everywhere;
\item[(iii)] there exists  a constant $\mu$ with $\mu\geq\mu_0>0$
such that for $|Q|<1$ we have
      $$\left|\int_{Q} a(x) dx\right|\leq  |Q|^{\mu}. $$
 \end{itemize}
\end{Def}

We now define the Hardy spaces $\mathcal H_{\mu_0}(\O)$
appeared in Theorem \ref{Thm}.

 \begin{Def}\label{def2.4}[Atomic Hardy spaces]
A function $f$ defined on $\O$ belongs to $\mathcal H_{\mu_0}(\O)$ if
the function $f$ can be expressed as
\begin{equation}\label{2.1}
f=\sum_{k} \lam_k a_k,
\end{equation}
where $a_k$ is a collection of $\mathcal H_{\mu_0}$ atoms with
respect to the cube $Q_k$ with $Q_k\subset \O$ and
$\lam_k$ is a sequence of complex numbers with $\sum|\lam_k|<\infty.$
Furthermore, the norm of $\|f\|_{\mathcal H_{\mu_0}}$ is defined by
$$\|f\|_{\mathcal H_{\mu_0}}=\inf \sum|\lam_k|,
$$
where  the infimum is taken over all the decompositions \eqref{2.1}.
 \end{Def}

We need to mention that when $\mu_0=1/3,$ Chang et al give an
equivalent definition of $\mathcal H_{\mu_0}$ by means of
a grand maximal function, see Definition 1.2 and Theorem 2.5 in \cite{CDS}.

\subsection{BMO spaces}
A local integrable function $f$ will be said to belong to
BMO if the inequality
\begin{equation}\label{bmo}
\frac{1}{|Q|} \int_{Q} |f(x)-f_{Q}| dx \leq A
\end{equation}
holds for all cubes $Q,$ here $f_{Q}=|Q|^{-1}\int_{Q} f dx$ denotes
the mean value of $f$ over the cube $Q.$
The smallest bound $A$ for which \eqref{bmo} is satisfied
is the taken to be the semi-norm of $f$ in this space,
and is denoted by $\|f\|_{BMO}.$

\begin{Lem} \cite[Sobolev embedding into BMO]{CP}\label{Lem2.5}
Let  $g\in BMO(\O)$ with $\nabla g\in{L(3, \infty)(\O)}$.
Then we have
\begin{equation}\label{2.3}
\|g\|_{BMO(\O)}\leq \|\nabla g\|_{L(3, \infty)(\O)}.
\end{equation}
\end{Lem}

\begin{Lem}[Duality of BMO]\label{Lem2.6}
Let $f\in \mathcal H_{\mu_0}(\O)$ and $g\in BMO(\O)$.
Then for any $1/\mu_0<p<\infty$ we have
\begin{equation}\label{2.4}
\int_{\O} f\cdot g dx
\leq C(\|g\|_{BMO(\O)}+\|g\|_{L^p(\O)})\|f\|_{\mathcal H_{\mu_0}(\O)},
\end{equation}
with a constant $C$ depending only on $\mu_0,$ $p$ and the domain $\O$.

\end{Lem}
\begin{proof}
From the definition of the space $\mathcal H_{\mu_0}(\O),$ the integral
in the left side of \eqref{2.4} can be written as
\begin{equation}\label{2.6}
\int_{\O} f\cdot g dx=\sum_{k}\lam_k\int_{Q_k}   a_k \cdot g dx,
\end{equation}
where $a_k$ is a collection of $\mathcal H_{\mu_0}$ atoms  and
$\lam_k$ is a sequence of complex numbers. Note that
\begin{equation}\label{2.5}
\int_{Q_k}   a_k \cdot g dx
=\int_{Q_k}   a_k \cdot (g-g_{Q_k}) dx+g_{Q_k}\int_{Q_k}   a_k  dx,
\end{equation}
where $g_{Q_k}$ denotes the mean value of $g$ over the cubic  $Q_k.$
From the definition of the  BMO space and
the condition (ii) in Definition \ref{def2.3}, it follows that
$$\left|\int_{Q_k}   a_k \cdot (g-g_{Q_k}) dx\right|\leq \|g\|_{BMO(\O)}.
$$
From the condition (iii) in Definition \ref{def2.3}
and by H\"{o}lder's inequality, we have
$$\left|g_{Q_k}\int_{Q_k}   a_k  dx\right|
\leq \left||Q_k|^{\mu-1}\int_{Q_k}   g  dx\right|
\leq C|Q_k|^{\mu_0-1+1/q} \|g\|_{L^p(\O)}
\leq C \|g\|_{L^p(\O)}
$$
where $p$ and $q$ are conjugate exponents satisfying $1/p+1/q=1,$ $p>1/\mu_0$ and
the constant $C$ depends on $p,$ $\mu_0$ and the domain.
Plugging the above two inequalities to \eqref{2.5},
and then by \eqref{2.6} we have
\begin{equation}\label{2.7}
\left|\int_{\O} f\cdot g dx\right|\leq
C(\|g\|_{BMO(\O)}+\|g\|_{L^p(\O)})\sum_{k}|\lam_k|.
\end{equation}
Taking the infimum on both sides in \eqref{2.7},
we obtain this lemma.
\end{proof}

\section{Proof of the main Theorem}\label{section4}

Before proving our main theorem, we first introduce the Dirichlet
fields $\mathbb{H}_2(\O):$
$$
\mathbb{H}_2(\O)=\{\u\in C^2(\O):~\dv
\u=0,\, \curl \u=0 \text{ in }\O, \, \nu\times\u=0\text{ on }\p\O \},
$$
 and the Neumann fields $\mathbb{H}_1(\O):$
$$
\mathbb{H}_1(\O)=\{\u\in C^2(\O):~\dv \u=0,\, \curl \u=0 \text{ in
}\O,\, \nu\cdot\u=0\text{ on }\p\O \}.
$$
Both of the spaces depend only  on the  topological structure of
$\O,$ and
$$ \dim(\mathbb{H}_2(\O))=m,\q \dim(\mathbb{H}_1(\O))=N,
$$
where $N$ and $m$  are respectively the first
and   the second  Betti number of the domain $\O$,
in this note we assume both of them are finite,
we refer to  \cite[Chapter 9]{DL} for details.

Let $1<p<\infty$ and $1\leq q\leq\infty.$
Then we define
$$
 V(p,q)\equiv\left\{\w\in L^{p}(\O), \nabla\w\in
L(p,q)(\O)~:~\dv\,\w=0,~ \nu\times\w=0 \text{ on }\p\O \right\},
$$
$$
 X(p,q)\equiv\left\{\w\in L^{p}(\O), \nabla\w\in
L(p,q)(\O)~:~\dv\,\w=0,~ \nu\cdot\w=0 \text{ on }\p\O \right\}.
$$

We now  establish the
decomposition for the vector fields on the Lorentz spaces.

\begin{Lem}[Decomposition of the Lorentz spaces]\label{Lem3.1}
 Suppose that $\O$ is a bounded domain in $\Bbb R^3$ with
  $C^2$ boundary,
 and let $1<p<\infty$ and $1\leq q\leq\infty.$
Then

Case 1. Each element $\u\in L(p,q)(\O)
$ has the unique
decomposition:
\begin{equation} \label{4.1}
 \u=\nabla v+\curl\w+ \h,
\end{equation}
where  $\nabla v\in L(p,q)(\O),$  $\w\in V(p,q)(\O)$
and $\h\in \mathbb{H}_1(\O).$
  Also, we have the estimate:
\begin{equation} \label{3.2}
\|\nabla v\|_{L(p,q)(\O)}+\|\nabla\w\|_{L(p,q)(\O)}
+\|\h\|_{L(p,q)(\O)}\leq C(p,q,\O)\|\u\|_{L(p,q)(\O)}.
\end{equation}

 Case 2. Each element $\u\in L(p,q)(\O)$ has the unique
decomposition:
\begin{equation} \label{4.3}
 \u=\nabla \hat{v}+\curl\hat{\w}+ \hat{\h}.
\end{equation}
where  $\nabla\hat{v}\in L(p,q)(\O),$ $\hat{v}\in W_0^{1,r}(\O)$
with $r<p$
and  $\nabla\hat{\w}\in X(p,q)(\O)$
 $\hat{\h}\in \mathbb{H}_2(\O).$
 Also, we have the estimate:
\begin{equation} \label{4.4}
\|\nabla\hat{v}\|_{L(p,q)(\O)}+\|\nabla\hat{\w}\|_{L(p,q)(\O)}
+\|\hat{\h}\|_{L(p,q)(\O)}\leq C(p,q,\O)\|\u\|_{L(p,q)(\O)}.
\end{equation}
\end{Lem}

\begin{proof}
The decompositions for vector fields and the estimate in the Sobolev $L^p$
spaces  have been obtained earlier by Kozono and Yanagisawa in \cite{KY}.
Hence, it suffices to show the estimate \eqref{3.2} and \eqref{4.4}.
We shall use the fact that
the Lorentz space $L(p,q)$ is the  real interpolation
 space between Lebesgue spaces
 $L^{p_1}$ and $L^{p_2}$
with $p_1<p<p_2$  to obtain the  estimate.
 We only prove the inequality \eqref{3.2},
 since \eqref{4.4} can be treated in a similar way.

 From  Simader and Sohr in \cite{SS1}
 we see that, for any  $\u\in L^p(\O)$ with $1<p<\infty$
 there exists $v_{\u}\in W^{1,p}$  being
  the weak solution of the Neumann problem and
  satisfying $\int v_{\u} dx=0$ such that
\begin{equation}\label{3.5}
\Delta v_{\u}=\dv\u\q\text{ \rm in } \O,
\quad\qq \frac{\p v_{\u}}{\p\nu}=\nu\cdot\u\q\text{ \rm on }\p\O.
\end{equation}
That is $v_{\u}$ satisfying the following weak form
$$\int_{\O} \nabla v_{\u}\cdot \nabla \phi dx
=\int_{\O} \u\cdot  \nabla \phi dx
\q \text{for any } \phi\in W^{1,p}(\O).
$$
Define the linear operator $T_1$ from $L^p(\O)$ to $L^{p}(\O)$  by
$$T_1: \u\to T_1\u=\nabla v_{\u}.$$
Then we can  get the estimate
$$ \|T_1 \u\|_{L^p(\O)}\leq C \|\u\|_{L^p(\O)}.
$$
Noting that the Lorentz space $L(p,q)(\O)$ can be
expressed by the real interpolation
between  $L^{p_1}$ and $L^{p_2}$
with $p_1<p<p_2$ (see \cite[Corollary 7.27]{AD}), and
then applying the interpolation theorem  (see \cite[section 7.23]{AD})
for any $\u\in L(p,q)(\O)$ we have
$$ \|T_1 \u\|_{L(p,q)(\O)}\leq C \|\u\|_{L(p,q)(\O)}.
$$

Similarly, there exists $\w_{\u}$ satisfying the weak form of the system
\begin{equation}\label{main system}
\begin {cases}
\curl\curl \w_{\u}=\curl\u & \text{ \rm in }\O,\\
\dv\w_{\u}=0 &\text{ \rm in }  \O,\\
\nu\times\w_{\u}=0 &\text{ \rm on }  \p\O.
\end {cases}
\end{equation}
That is
$$\int_{\O}\curl\w_{\u}\cdot\curl\mathbf\Psi dx
=\int_{\O}  \u\cdot\curl\mathbf\Psi dx \q \text{for any } \mathbf\Psi\in V(p,p)(\O).
$$

Define the linear operator $T_2$ from $L^p(\O)$ to $L^p(\O)$  by
$$T_2: \u\to T_2\u=\nabla\w_{\u}.$$
Then we can  get the estimate
$$ \|T_2 \u\|_{L^p(\O)}\leq C \|\u\|_{L^p(\O)}.
$$
Applying the interpolation theorem, for any $\u\in L(p,q)(\O)$ we have
$$ \|T_2 \u\|_{L(p,q)(\O)}\leq C \|\u\|_{L(p,q)(\O)}.
$$

The estimate
$$ \|\h_\u\|_{L(p,q)(\O)}\leq C \|\u\|_{L(p,q)(\O)}.
$$
is directly from the fact that
 $\h_{\u}$ can be expressed by
$$\h_{\u}=\sum_{i=1}^N (\u, \h_i)\h_i,
$$
where $\h_i\in \mathbb{H}_1(\O).$
Thus we get the estimate \eqref{3.2} and the proof is now complete.

\end{proof}

\begin{Lem}\label{Lem3.2}

 Under the assumption in Lemma \ref{Lem3.1},
 for any vectors $\u\in L(3,\infty)(\O),$

Case 1. the decomposition \eqref{4.1} holds
and for $1\leq p<\infty$ we have the estimate
\begin{equation} \label{3.7}
\|v\|_{L^{p}(\O)}+\|\w\|_{L^{p}(\O)}\leq C(p,
\O)\|\u\|_{L(3,\infty)(\O)}.
\end{equation}

 Case 2. the decomposition \eqref{4.3} holds
 and for $1\leq p<\infty$ we have the estimate
\begin{equation} \label{3.8}
\|\hat{v}\|_{L^{p}(\O)}+\|\hat{\w}\|_{L^{p}(\O)}\leq C(p,
\O)\|\u\|_{L(3,\infty)(\O)}.
\end{equation}
\end{Lem}
\begin{proof}
It suffices to prove the estimate \eqref{3.7}
and \eqref{3.8}. We first prove \eqref{3.7}.
Applying the estimate in \cite[Proposition 2.1]{KY2}
and from Lemma \ref{Lem2.1} (ii),
we see that for  $p_1$ with $p_1<3$ we have
$$\|v\|_{L^{p_1}(\O)}\leq C(p_1,\O)\|\u\|_{L^{p_1}(\O)}
\leq C(p_1,\O)\|\u\|_{L(3,\infty)(\O)}.
$$
Then applying Lemma \ref{Lem2.1} (ii) for $\nabla v$ and
by the inequality \eqref{3.2}, we have
$$\|\nabla v\|_{L^{p_1}(\O)}
\leq C(p_1,\O)\|\nabla v\|_{L(3,\infty)(\O)}
\leq C(p_1,\O)\|\u\|_{L(3,\infty)(\O)}.
$$
Since $W^{1,p_1}$ is continuously embedded into
the space $L^p$ with $p<3 p_1/(3-p_1)$,
for any $1<p<\infty$ we can choose suitable $p_1$ such
that the following holds
$$\|v\|_{L^{p}(\O)}\leq C(p,\O)\|v\|_{W^{1,p_1}(\O)}
\leq C(p,\O)\|\u\|_{L(3,\infty)(\O)}.
$$
Other inequalities can be obtained by a similar way.
\end{proof}

\v0.2in
We are now in the position to prove our main theorem.
\begin{proof}[Proof of Theorem \ref{Thm}]
We first prove the case where $\nu\cdot\u=0.$
For $\u\in C^{1,\alpha}(\bar{\O})$ we take the decomposition
$$
 \u=\nabla p_{\u}+\curl \w_{\u}+ \h_{\u},
$$
where\,$p_{\u}\in C^{2,\alpha}(\O)$   satisfying
\begin{equation}\label{4.7}
\Delta p_{\u}=\dv\u\q\text{ \rm in } \O,
\quad\qq \frac{\p p_{\u}}{\p\nu}=0\q\text{ \rm on }\p\O;
\end{equation}
$\w_{\u}\in V^{2,\alpha}_{\tau}(\O)$ with
$$
 V^{2,\alpha}_{\tau}(\O)\equiv\left\{\w\in
C^{2,\alpha}(\O)~:~\dv\,\w=0\text{ in }\O,~
\nu\times\w=0 \text{ on }\p\O \right\}
$$
 and $\w_{\u}$ satisfying
\begin{equation}\label{3.10}
\begin {cases}
\curl\curl \w_u=\curl\u & \text{ \rm in }\O,\\
\dv\w_u=0 &\text{ \rm in }  \O,\\
\nu\times\w_{\u}=0 &\text{ \rm on }  \p\O
\end {cases}
\end{equation}
and $\h_{\u}\in
\mathbb{H}_1(\O).$

Let $p_{\u}$ be defined by \eqref{4.7}.
For any vector $\F\in L(3,\infty)(\O),$ we use
 the decomposition \eqref{4.1} in Lemma \ref{Lem3.1}
 for the vector $\F$, then
$$
(\nabla p_{\u},\F)=(\nabla p_{\u},\nabla p_{\mathbf F})
=(\dv\u, p_{\mathbf F}).
$$
From Lemma \ref{Lem2.2} and the above equality we see that
\begin{equation}\label{3.11}
\|\nabla p_{\u}\|_{L(3/2,1)(\O)}
\leq C\sup_{\mathbf F\in L(3,\infty)(\O)}
 \frac{(\nabla p_{\u}, \mathbf F)}{\|\mathbf F\|_{L(3,\infty)(\O)}}
 \leq  C\sup_{\mathbf F\in L(3,\infty)(\O)}
 \frac{(\dv\u, p_{\mathbf F})}{\|\mathbf F\|_{L(3,\infty)(\O)}}.
\end{equation}
The duality (\ref{Lem2.6}), for any $1/\mu_0<p<\infty,$ implies that
$$
|(\dv\u, p_{\mathbf F})|\leq C(\mu_0,p,\O)(\|p_{\mathbf F}\|_{BMO(\O)}
+\|p_{\mathbf F}\|_{L^p(\O)})\|\dv\u\|_{\mathcal H_{\mu_0}(\O)}.
$$
The inequalities \eqref{3.2} and \eqref{3.7} show that
$$\|\nabla p_{\mathbf F}\|_{L(3,\infty)(\O)}
+\| p_{\mathbf F}\|_{L^{p}(\O)}\leq C(p,\O)\|\mathbf F\|_{L(3,\infty)(\O)}.
$$
Let $p$ now be fixed. Then from the above two inequalities
and by the inequality in Lemma \ref{Lem2.5}, we have
\begin{equation}\label{3.12}
|(\dv\u, p_{\mathbf F})|
\leq C(\mu_0,\O)\|\mathbf F\|_{L(3,\infty)(\O)}
\|\dv\u\|_{\mathcal H_{\mu_0}(\O)},
\end{equation}
The inequalities \eqref{3.11} and \eqref{3.12} give
\begin{equation}\label{4.9}
\|\nabla p_{\u}\|_{L(3/2,1)(\O)}
\leq C(\mu_0,\O) \|\dv \u\|_{\mathcal H_{\mu_0}(\O)}.
\end{equation}

Let $\w_{\u}$ be defined by \eqref{3.10}. Using the decomposition \eqref{4.1}
 in Lemma \ref{Lem3.1} for the vector $\mathbf\Phi$, we have,  by duality,
 $$
\| \curl \w_{\u}\|_{L(3/2,1)(\O)}
\leq C\sup_{\mathbf \Phi\in L(3,\infty)(\O)}
\frac{(\curl\w_{\u},\mathbf\Phi)}{\|\mathbf\Phi\|_{L(3,\infty)(\O)}}
\leq C\sup_{\mathbf \Phi\in L(3,\infty)(\O)}
 \frac{(\curl\u,\w_{\mathbf \Phi})}{\|\mathbf \Phi\|_{L(3,\infty) (\O)}}.
$$
Similar to the estimate for $\nabla p_{\u}$,
we apply Lemma \ref{Lem2.6}, Lemma \ref{Lem3.1} and
Lemma \ref{Lem3.2} and get
\begin{equation}\label{4.11}
\| \curl \w_{\u}\|_{L(3/2,1)(\O)}
\leq C(\mu_0,\O)\|\curl \u\|_{\mathcal H_{\mu_0}(\O)}.
\end{equation}

Since $\h_{\u}$ can be expressed by
$$\h_{\u}=\sum_{i=1}^N (\u, \h_i)\h_i,
$$
where $\h_i\in \mathbb{H}_1(\O),$
then combining the inequalities \eqref{4.9} and  \eqref{4.11}
we get \eqref{estimate}.

We now prove the case where $\nu\times\u=0$.
From \cite[Theorem 2.1]{KY} we see that
for every $\hat{\u}\in C^{1,\alpha}(\O)$ there
 exists a decomposition
\begin{equation}\label{4.12}
 \hat{\u}=\nabla \hat{p}_{\u}+\curl \hat{\w}_{\u}+ \hat{\h}_{\u},
\end{equation}
where
$\hat{p}_{\u}\in C^{2,\alpha}(\O) $ satisfying
\begin{equation}\label{4.13}
\Delta \hat{p}_{\u}=\dv\hat{\u}\q\text{ \rm in } \O,
\quad\qq \hat{p}_{\u}=0\q\text{ \rm on }\p\O;
\end{equation}
$\hat{\w}_{\u}\in X^{2,\alpha}_{n}(\O)$ with $X^{2,\alpha}_{n}$ defined by
$$
 X^{2,\alpha}_{n}(\O)\equiv\left\{\w\in
C^{2,\alpha}(\O)~:~\dv\,\w=0,~ \nu\cdot\w=0 \text{ on }\p\O \right\}
$$
and $\hat{\w}_{\u}$ satisfying
\begin{equation}\label{4.14}
\begin {cases}
\curl\curl \hat{\w}_u=\curl\hat{\u} & \text{ \rm in }\O,\\
\dv\hat{\w}_u=0 &\text{ \rm in }  \O,\\
\nu\times\curl\hat{\w}_{\u}=\nu\times\hat{\u}=0 &\text{ \rm on }  \p\O,\\
\nu\cdot\hat{\w}_{\u}=0 &\text{ \rm on }  \p\O
\end {cases}
\end{equation}
and $\hat{\h}_{\u}\in
\mathbb{H}_2(\O).$

We shall estimate each term in \eqref{4.12}.
Let $\hat{p}_{\u}$ be defined by \eqref{4.13}.
Using the decomposition \eqref{4.3}
 in lemma \ref{Lem3.1} for the vector $\mathbf A$,
 we have
$$(\nabla \hat{p}_{\u},\mathbf A)=(\nabla \hat{p}_{\u},
\nabla\hat{\mathbf v}_{\mathbf A})
=(\dv \hat{\u},\hat{\mathbf v}_{\mathbf A}).
$$
By the estimate \eqref{4.4}, we have
$$\|\nabla \hat{p}_{\u}\|_{L(3/2,1)(\O)}\leq
C\sup_{\mathbf A\in L(3,\infty)(\O),\mathbf A\neq0}
\frac{(\nabla \hat{p}_{\u},\mathbf A)}{\|\mathbf A\|_{L(3,\infty)(\O)}}
\leq C\sup_{\mathbf A\in L(3,\infty)(\O),\mathbf A\neq0}
\frac{(\dv \hat{\u},\hat{\mathbf v}_{\A})}{\|\mathbf A\|_{L(3,\infty)(\O)}}.$$
Similar to the estimate for $\nabla p_{\u}$, it follows that
$$\|\nabla \hat{p}_{\u}\|_{L(3/2,1)(\O)}
\leq C(\mu_0, \O) \|\dv \hat{\u}\|_{\mathcal {H}(\O)}.
$$

For the second term in the right side of \eqref{4.12},
using the decomposition \eqref{4.3}
 in Lemma \ref{Lem3.1} for the vector $\mathbf\Phi,$ we have
 $$(\curl\hat{\w}_{\u}, \mathbf\Phi)=
 (\curl\hat{\w}_{\u},\curl\hat{\w}_{\mathbf\Phi})=
 (\curl\hat{\u},\hat{\w}_{\mathbf\Phi}).
 $$
 Similar to the estimate for $\curl \w_{\u}$ (see \eqref{4.11}),
 from the above equality we get
$$
\| \curl \hat{\w}_{\u}\|_{L(3/2,1)(\O)}\leq C\sup_{\mathbf \Phi\in L(3,\infty)(\O),
\mathbf\Phi\neq0} \frac{|(\curl\hat{\w}_{\u},\mathbf\Phi)|}
{\|\mathbf\Phi\|_{L(3,\infty)(\O)}}
\leq C(\mu_0,\O) \|\curl \hat{\u}\|_{\mathcal {H}_{\mu_0}(\O)}.
$$

Since we have
$$\hat{\h}_{\u}=\sum_{i=1}^m (\hat{\u}, \hat{\h}_i)\hat{\h}_i,
$$
where $\hat{\h}_i\in \mathbb{H}_2(\O),$
combing the above two inequalities we get the estimate \eqref{estimate}.
\end{proof}

In view of the proof of Theorem \ref{Thm},  we can easily  see that
\begin{Cor}  Under the assumption in Theorem
\ref{Thm}, if either
\begin{itemize}
\item[(i)]  $\nu\cdot\u=0$ on $\p\O$ and the first Betti number $N=0$; or
\item[(ii)] $\nu\times\u=0$ on $\p\O$ and the second Betti number $m=0$,
 \end{itemize}
then it holds that
$$
\| \u\|_{L(3/2,1)(\O)}
\leq C(\mu_0,\O)\left(\|\dv\u\|_{\mathcal {H}_{\mu_0}(\O)}
+\|\curl\u\|_{\mathcal {H}_{\mu_0}(\O)}\right),
$$
where the constant $C$ depends only on $\mu_0$
and the domain $\O$ but not on the vector $\u.$
\end{Cor}

 \vspace {0.5cm}

\subsection*{Acknowledgments.}
The work of the first author was partly supported by Japan Society for the Promotion of Science through grants Kiban (S) 21224001, Kiban (A) 23244015 and Houga 25610025. The work of the second author was partly supported by the National Natural Science Foundation of China through grant no.\ 11171111.


\vspace {0.5cm}

\begin {thebibliography}{DUMA}



\bibitem[1]{AD}  R.A. Adams,
{\it  Sobolev Spaces,}
Academic Press, New York, 1975.

\bibitem[2]{BB1} J. Bourgain, H. Brezis,
{\it  On the equation $\dv Y = f$
and application to control of phases,}
J. Amer. Math. Soc. {\bf 16} (2), (2002) 393-426.

\bibitem[3]{BB2} J. Bourgain, H. Brezis,
{\it  New estimates for the Laplacian,
the div-curl, and related Hodge systems,}
C. R. Math. Acad. Sci. Paris {\bf 338}, (2004) 539-543.

\bibitem[4]{BB} J. Bourgain, H. Brezis,
{\it  New estimates for elliptic
equations and Hodge type systems,}
J. Eur. Math. Soc.  {\bf 9} (2), (2007) 277-315.

\bibitem[5]{BV1} H. Brezis, J. Van Schaftingen,
 {\it Boundary estimates for elliptic systems with $L^1$-data,}
Calc. Var. Partial Diff. Eq. {\bf 30} (3), (2007) 369-388.

\bibitem[6]{CDS} D-C. Chang, G. Dafni, E.M. Stein,
{\it Hardy spaces, BMO and boundary value problems
for the Laplacian on a smooth domain in $\mathbb{R}^N$,}
Trans. Amer. Math. Soc. {\bf 351}(4), (1999) 1605-1661.

\bibitem[7]{CP} A. Cianchi and L. Pick,
{\it Sobolev embeddings into BMO, VMO and $L^{\infty}$,}
Ark. Mat. {\bf 36}(2), (1998) 317-340.

\bibitem[8]{CW} M. Cwikel,
{The dual of weak $L^p$,}
Ann. Int. Fourier. {\bf 25}, (1975) 81-126.

\bibitem[9]{DL}  R. Dautray, J.L. Lions,
{\it Mathematical analysis and numerical
methods for science and technology,}
vol. {\bf 3}, Springer-Verlag, New York, 1990.

\bibitem[10]{WF}  W.G. Faris,
{\it Weak Lebesgue spaces and Quantum mechanical binding,}
Duke Math. J. {\bf 43} (2), (1976) 365-373.

\bibitem[11]{HU} R.A. Hunt,
{\it  On L(p,q) spaces,}
L'Enseignment Math., {\bf 12}, (1996) 249-276.

\bibitem[12]{KY}  H. Kozono,  T. Yanagisawa,
{\it $L^r$-variational inequality for vector
fields and the Helmholtz-Weyl decomposition in bounded domains,}
Indiana Univ. Math. J. {\bf 58}(4), (2009) 1853-1920.

\bibitem[13]{KY2}  H. Kozono, T. Yanagisawa,
{\it Global Div-Curl lemma on a bounded domains in $\mathbb{R}^3$,}
Journal of Functional Analysis  {\bf 256},  (2009) 3847-3859.

\bibitem[14]{LS} L. Lanzani,  E.M. Stein,
{\it A note on div curl inequalities,}
 Math. Res. Lett. {\bf 12}(1), (2005) 57-61 .

\bibitem[15]{VM} V. Maz��ya,
{\it Estimates for differential operators of vector analysis
involving $L^1$-norm,}
J. Eur. Math. Soc.  {\bf 12} (1), (2010) 221-240.

\bibitem[16]{MM} I. Mitrea, M. Mitrea,
{\it A remark on the regularity of the div-curl system,}
 Proc. Amer. Math. Soc. {\bf 137}, (2009) 1729-1733.

\bibitem[17]{SS1}  C.G. Simader, H. Sohr,
{\it A new approach to the Helmholtz decomposition
and the Neumann problem in $L^q$ -spaces for bounded and exterior
domains,}
In: G.P. Galdi (Ed.), Mathematical Problems Relating to
the Navier�CStokes Equations, in: Ser. Adv. Math. Appl. Sci., World
Scientific, Singapore, New Jersey, London, Hong Kong, 1992, pp.
1-35.

\bibitem[18]{SS2}  C.G. Simader,  H.  Sohr,
{\it The Dirichlet Problem for the Laplacian in
Bounded and Unbounded Domains,}
Pitman Res. Notes Math. Ser., vol.
{\bf 360}, Longman, 1996.

\bibitem[19]{ST} E.M. Stein,
{\it Harmonic Analysis: Real Variable Methods,
Orthogonality, and Oscillatory Integrals,}
Princeton University Press, Princeton NJ, 1993.

\bibitem[20]{VS1} J. Van Schaftingen,
 {\it Estimates for $L^1$ vector fields,}
 C. R. Math. Acad. Sci. Paris {\bf 339}, (2004) 181-186.

\bibitem[21]{VS2} J. Van Schaftingen,
{\it Estimates for $L^1$ vector fields with a second order condition,}
Acad. Roy. Belg. Bull. Cl. Sci. {\bf 15}, (2004) 103-112.

\bibitem[22]{VS3} J. Van Schaftingen,
{\it Estimates for $L^1$ vector fields under higher-order differential conditions,}
J. Eur. Math. Soc. {\bf 10}(4), (2008) 867-882.

\bibitem[23]{VS4} J. Van Schaftingen,
{\it Limiting fractional and Lorentz space estimates of differential forms,}
 Proc. Amer. Math. Soc. {\bf 138}(1), (2010) 235-240.

\bibitem[24]{W1} W. von Wahl,
{\it Estimating $\nabla\u$ by $\dv\u$ and $\curl\u$,}
 Math. Methods Appl. Sci. {\bfseries15}, (1992) 123-143.


\bibitem[25]{Xiang} X.F. Xiang,
{\it $L^{3/2}$-Estimates of vector fields with
$L^1$ curl in a bounded domain,}
Calc. Var. Partial Diff. Eq. {\bf 46} (1), (2013) 55-74.

\end{thebibliography}
\end {document}